\documentclass[11pt]{article}
\usepackage{amsmath,amsthm,amssymb,mathtools}
\usepackage{color}

\usepackage[hyperfootnotes=false]{hyperref}

\usepackage{pxfonts}

\usepackage{titlesec}
\titleformat{\paragraph}
{\normalfont\normalsize\bfseries}{\theparagraph}{1em}{}
\titlespacing*{\paragraph}
{0pt}{3.25ex plus 1ex minus .2ex}{1.5ex plus .2ex}

\def\titlerunning#1{\gdef\titrun{#1}}
\makeatletter
\def\author#1{\gdef\autrun{\def\and{\unskip, }#1}\gdef\@author{#1}}
\def\address#1{{\def\and{\\\hspace*{18pt}}\renewcommand{\thefootnote}{}%
\footnote {#1}}%
\markboth{\autrun}{\titrun}}
\makeatother
\def\email#1{\hspace*{4pt}{\em e-mail}: #1}
\def\MSC#1{{\renewcommand{\thefootnote}{}%
\footnote{\emph{Mathematics Subject Classification (2020):} #1}}}
\def\keywords#1{\par\medskip
\noindent\textbf{Keywords:} #1}

\newtheorem{theorem}{Theorem}[section]
\newtheorem{prop}[theorem]{Proposition}
\newtheorem{cor}[theorem]{Corollary}
\newtheorem{lemma}[theorem]{Lemma}

\theoremstyle{definition}

\newtheorem{remark}[theorem]{Remark}

\numberwithin{equation}{section}

\newcommand{\floor}[1]{\left\lfloor #1 \right\rfloor}

\def\0{\mathbf 0}
\def\w{\omega}

\def\cC{\mathcal C}

\def\cE{\mathcal E}

\def\cG{\mathcal G}
\def\cH{\mathcal H}

\def\cO{\mathcal O}

\def\cQ{\mathcal Q}

\def\cV{\mathcal V}

\def\cW{\mathcal W}
\def\cX{\mathcal X}
\def\cY{\mathcal Y}

\def\U{\mathrm U}

\def\AG{{\rm AG}}
\def\PG{{\rm PG}}

\def\F{{\mathbb F}}

\def\PGL{{\rm PGL}}

\def\U{{\rm U}}

\def\diag{{\rm diag}}
\def\rk{{\rm rk}}

\begin{document}


\baselineskip=16pt


\titlerunning{}

\title{$(r, s)$-sets from Desarguesian ovoids}

\author{Francesco Pavese}

\date{}

\maketitle

\address{F. Pavese: Dipartimento di Meccanica, Matematica e Management, Politecnico di Bari, Via Orabona 4, 70125 Bari, Italy; \email{francesco.pavese@poliba.it}
}

\bigskip

\MSC{Primary 51E20; Secondary 51E21; 51E22; 94B05.}


\begin{abstract}
An {\em $(r, s)$-set} in $\PG(n, q)$ is a set of points, say $\cX$, such that each $s$-dimensional projective subspace contains at most $r$ points of $\cX$. We investigate $(n, n-2)$-sets and $(n-2, n-3)$-sets in $\PG(n, q)$, $n \le 6$. We show that the trivial upper bounds on $(n, n-2)$-sets in $\PG(n, q)$, $4 \le n \le 6$, $(4, 3)$-sets in $\PG(6, q)$ and $(3, 2)$-sets in $\PG(5, q)$ are essentially sharp. A $(3, 2)$-set in $\PG(13, q)$ of size $\frac{q^6-1}{q-1}$ is also constructed.

\end{abstract}

\section{Introduction}

Let $q$ be a prime power and let $\PG(n, q)$ or $\AG(n, q)$ denote the $n$-dimensional projective or affine space over the finite field $\F_q$. A {\em cap} is a set of points in $\PG(n, q)$ such that at most two of them are on a line, whereas a set of points in $\PG(n, q)$ such that at most $n$ in a hyperplane is known as an {\em arc}. These objects have been extensively studied due to their connections to coding theory; see, for instance, \cite{HS}. More generally, following \cite{H1983}, a pointset $\cX$ in $\PG(n, q)$ is called {\em $(|\cX|; r, s, n, q)$-set} or simply {\em $(r, s)$-set} if the properties below are satisfied:
\begin{itemize}
\item[{\em i)}] each $s$-dimensional projective subspace contains at most $r$ points of $\cX$;
\item[{\em ii)}] $\cX$ spans the whole $\PG(n, q)$;
\item[{\em iii)}] there is an $(s+1)$-dimensional projective subspace containing $r+2$ points of $\cX$.
\end{itemize}
The term {\em $(r+2)$-general set}, $1 \le r \le n-1$, is also used to denote a $(r+1, r)$-set, see \cite{Bennett, TW, P}. Indeed, $\cX$ is an $(r+2)$-general set if any $r+2$ distinct points of $\cX$ are in general position. Hence a cap is a $3$-general set and an arc is an $(n+1)$-general set. An $(r, s)$-set can be similarly defined in $\AG(n, q)$. In particular, an $(r, s)$-set of $\AG(n, q)$ is also an $(r, s)$-set of $\PG(n, q)$. On the other hand, $(r, s)$-sets in $\AG(n, q)$ are also known as {\em $(s, r)$-subspace evasive sets} \cite{ST}. Recently these sets gained renewed interest due to their connections with bipartite Ramsey graphs \cite{PR} and list-decodable codes \cite{Gu}. Nonetheless, few explicit constructions of large $(r, s)$-sets are known in literature. In \cite{DL} the authors showed that, if $q$ is large enough, there exists an $(n^s, s)$-set in $\AG(n, q)$ of size $q^{n-s}$, whereas in \cite{IV} an $(n,1)$-set of size $(q-1)^{n-1}$ in $\PG(n, q)$ is exhibited. Further results can be found in \cite{BS, VV}. 

In this paper, we focus on $(n, n-2)$-sets in $\PG(n, q)$, $n = 4, 5, 6$, and on $(n-2, n-3)$-sets, i.e., $(n-1)$-general sets, in $\PG(n, q)$, $n = 5, 6$. A trivial upper bound shows that in $\PG(n, q)$ the size of an $(n, n-2)$-set cannot exceed $\left(n!\right)^{\frac{1}{n-1}} q^2 + O(q)$, whereas a $(4, 3)$-set and a $(3, 2)$-set, i.e., $5$-general set and $4$-general set, have size at most $\sqrt{2} q^{\frac{n-2}{2}} + O(q^{\frac{n-4}{2}})$ and $\sqrt{2} q^{\frac{n-1}{2}} + O(q^{\frac{n-3}{2}})$, respectively. Here we show the existence of a transitive set of size $q^2-q+1$ in $\AG(6, q)$ that is a $(6,4)$-set and a $(4, 3)$-set. Such a set is obtained by considering a suitable hyperplane section of the so-called {\em Desarguesian (partial) ovoid} of $\PG(7, q)$. By projecting an $(r, s)$-set of $\PG(n, q)$ from one of its points, an $(r-1, s-1)$-set of $\PG(n-1, q)$ arises. It follows that the trivial upper bounds on $(n, n-2)$-sets in $\PG(n, q)$, $4 \le n \le 6$, $(4, 3)$-sets in $\PG(6, q)$ and $(3, 2)$-sets in $\PG(5, q)$ are essentially sharp. To the best of our knowledge, besides a $4$-general set in $\PG(5, q)$ of size $q^2+1$ described by Cooperstein in \cite[Theorem~7.7]{Co}, no instances of $(r, s)$-sets with these cardinalities were previously known in the literature. Our main results are summarized in the following:

\begin{theorem}
\begin{enumerate}
\item[i)] In $\AG(6, q)$, $q \ge 4$, there exists a transitive set of size $q^2-q+1$ that is a $(4, 3)$-set and a $(6, 4)$-set.
\item[ii)] In $\AG(5, q)$, $q \ge 4$, there exists a set of size $q^2-q$ that is a $(3, 2)$-set and a $(5, 3)$-set.
\item[iii)] In $\PG(5, q)$, $q \ge 4$, there exists a set of size $q^2-q+2$ that is a $(3, 2)$-set.
\item[iv)] In $\PG(13, q)$ there exists a transitive set of size $\frac{q^6-1}{q-1}$ that is a $(3, 2)$-set.
\end{enumerate}
\end{theorem}

In Section~\ref{bb}, simple upper bounds on the size of the largest $(n, n-2)$-sets and $(4, 3)$-sets in $\PG(n, q)$ are discussed. In Section~\ref{pre}, crucial properties of the Desarguesian (partial) ovoid $\cO$ of $\PG(7, q)$ are established. In Section~\ref{pg6}, it is shown that if a hyperplane $\cH$ of $\PG(7, q)$ intersects $\cO$ in $q^2-q+1$ points, then $\cH \cap\cO$ is a $(4, 3)$-set and a $(6, 4)$-set of $\cH$. The set obtained by projecting $\cH \cap \cO$ from one of its points is further investigated in Section~\ref{pg5}. Finally, in Section~\ref{pg13}, a transitive $(3, 2)$-set of $\PG(13, q)$ of size $\frac{q^6-1}{q-1}$ is constructed.


\section{Upper bounds}\label{bb}

In this section, by using a counting argument, upper bounds on the size of the largest $(n, n-2)$-sets and $(4, 3)$-sets in $\PG(n, q)$ are presented. 

\begin{prop}\label{upper1}
Let $\cX$ be an $(n, n-2)$-set of $\PG(n, q)$, then 
\begin{align*}
|\cX| \le \left(\frac{n! (q^{n+1}-1)(q^n-1)}{(q-1)(q^2-1)}\right)^{\frac{1}{n-1}} + n-2.
\end{align*}
\end{prop}
\begin{proof}
Let $\cX$ be a $(n, n-2)$-set of $\PG(n, q)$. Every subset of size $n-1$ of $\cX$ lies in at least an $(n-2)$-dimensional projective subspace of $\PG(n, q)$. On the other hand, every $(n-2)$-dimensional projective subspace of $\PG(n, q)$ contains at most $n$ points of $\cX$ and hence at most $n$ subsets of $\cX$ of size $(n-1)$. Therefore
\begin{align*}
{|\cX| \choose n-1} \le n \frac{(q^{n+1}-1)(q^n-1)}{(q-1)(q^2-1)}.
\end{align*}
Since 
\begin{align*}
\frac{(|\cX| - n + 2)^{n-1}}{(n-1)!} \le {|\cX| \choose n-1},
\end{align*}
the claim follows. 
\end{proof}

\begin{cor}
An $(n, n-2)$-set of $\PG(n, q)$ has at most $\left(n!\right)^{\frac{1}{n-1}} q^2 + O(q)$ points.
\end{cor}

By adding one more point to a cap of size $q^2+1$ of $\PG(3, q)$, a $(3, 1)$-set matching the theoretical upper bound up to a constant factor is obtained. On the other hand, no examples of order $q^2$ are known if $n > 3$.

\begin{prop}\label{upper2}
Let $\cX$ be a $5$-general set of $\PG(n, q)$, then 
\begin{align*}
|\cX| \le \frac{\sqrt{8q^n+q^2-6q+1} + 3q-5}{2(q-1)}.
\end{align*}
\end{prop}
\begin{proof}
Let $\cX$ be a $5$-general set of $\PG(n, q)$. A point of $\PG(n, q) \setminus \cX$ lies on at most one line secant to $\cX$. On the other hand, if a point $P$ outside the lines secant to $\cX$ belongs to two planes $\pi_1$, $\pi_2$, spanned by points of $\cX$, where $\pi_i \cap \cX = \{A_i, B_i, C_i\}$, then the two triples $\{A_1, B_1, C_1\}$ and $\{A_2, B_2, C_2\}$ must be disjoint, otherwise the two planes would have a line in common and they would generate a solid containing five points of $\cX$, which is impossible. Thus, through a point $P$ outside the secants there can pass at most $\floor{\frac{|\cX|}{3}}$ planes spanned by points of $\cX$, since the corresponding triples must be pairwise disjoint. It follows that
\begin{align*}
|\cX|+(q-1)\frac{|\cX|(|\cX|-1)}{2} + (q-1)^2 \frac{|\cX|(|\cX|-1)(|\cX|-2)}{6 {\floor{|\cX|/3}}} \leq \frac{q^{n+1}-1}{q-1}.
\end{align*}
Consequently, $\cX$ also satisfies
\begin{align*}
 |\cX|+(q-1)\frac{|\cX|(|\cX|-1)}{2} + (q-1)^2 \frac{(|\cX|-1)(|\cX|-2)}{2} \leq \frac{q^{n+1}-1}{q-1},
\end{align*}
since 
\begin{align*}
\frac{(|\cX|-1)(|\cX|-2)}{2} \le \frac{|\cX|(|\cX|-1)(|\cX|-2)}{6 {\floor{|\cX|/3}}},
\end{align*}
that is 
\begin{align*}
\frac{q(q-1)}{2} |\cX|^2+\frac{q(5-3q)}{2} |\cX| + (q-1)^2 \leq  \frac{q^{n+1}-1}{q-1},
\end{align*}
from which the statement follows.
\end{proof}

\begin{cor}
A $(4, 3)$-set of $\PG(n, q)$ has at most $\sqrt{2} q^{\frac{n-2}{2}} + O(q^{\frac{n-4}{2}})$ points.
\end{cor}

Note that equality in the bound of Proposition~\ref{upper2} is attained if $(n, q) \in \{(4, 2), (5, 3)\}$. Indeed, a frame in $\PG(4, 2)$ is a $5$-general set of size $6$, whereas it is known that there exists a unique $5$-general set in $\PG(5, 3)$ of size $12$. See also \cite[p. 288]{G}. In these cases, by considering the points of $\cX$ as the columns of a parity check matrix of a code, either the $[6, 1, 6]_2$ repetition code or the $[12, 6, 6]_3$ extended Golay code is obtained. Moreover, the bound is known to be sharp if $n = 4$. Indeed, in this case a $5$-general set is an arc, hence $|\cX| \le q+1$ by \cite[Theorem 6.40]{HT}, and normal rational curves of $\PG(4, q)$ provide instances of such sets. In what follows we show that the bound turns out to be essentially sharp also in the case when $n = 6$. 

By \cite[Proposition~3.1]{P}, the size of a $4$-general set of $\PG(n, q)$ cannot exceed $\frac{\sqrt{8q^{n+1} + q^2 -6q +1} + q - 3}{2(q-1)}$.

\begin{cor}
A $(3, 2)$-set of $\PG(n, q)$ has at most $\sqrt{2} q^{\frac{n-1}{2}} + O(q^{\frac{n-3}{2}})$ points.
\end{cor}

\section{Preliminary results}\label{pre}

Consider the symplectic polar space $\cW(7, q)$ or the orthogonal polar space $\cQ^+(7, q)$ of $\PG(7, q)$. A (partial) ovoid in $\cQ^+(7, q)$ (resp. $\cW(7,q)$) is a subset of (at most) $q^3+1$ (resp. $q^4+1$) pairwise non-perpendicular points. In \cite{K} Kantor constructed an ovoid of $\cQ^+(7, q)$ for $q$ even. In \cite{L1} Lunardon showed that the same set, independently from $q$, is the Grassmann embedding of a Desarguesian plane-spread of $\PG(5, q)$ (see also \cite{L}). In \cite{C} Cossidente proved that, if $q$ is odd, then it is a maximal partial ovoid of $\cW(7, q)$. We refer to this set as the {\em Desarguesian (partial) ovoid} of $\PG(7, q)$. Here we characterize the twisted cubics embedded in a Desarguesian (partial) ovoid and prove that if seven distinct points of a Desarguesian (partial) ovoid are contained in a four-dimensional projective subspace of $\PG(7, q)$, then four of them are contained in one of its twisted cubics. 

In $\F_{q^{3}}^{8}$ consider the $8$-dimensional $\F_q$-subspace $U_1$ given by the set of vectors
\begin{align*}
\left\{ \left(a, b^{q^2}, b^{q}, c, b, c^{q}, c^{q^2}, d \right) : a, d \in \F_{q}, b, c \in \F_{q^3} \right\}.
\end{align*}
Then $\PG(U_1)$ is a $7$-dimensional projective space over $\F_q$, since $\dim(U_1) = 8$. For $(a,b,c,d)\neq (0,0,0,0)$ denote by $P(a,b,c,d)$ the point of $\PG(U_1)$ defined by the vector $\left(a, b^{q^2}, b^{q}, c, b, c^{q}, c^{q^2}, d \right)$. Consider the following set consisting of $q^3+1$ points of $\PG(U_1)$ 
\begin{align*}
\cO_1 = \left\{ P(1,t,t^{q^2+q},t^{q^2+q+1}) = (1,t) \otimes (1,t^q) \otimes (1,t^{q^2}) \mid t \in \F(q^3) \right\} \cup \{P(0,0,0,1) = (0,1) \otimes (0,1) \otimes (0,1) \}.
\end{align*}
Thus points of $\cO_1$ are in bijection with the $q^3+1$ points of $\PG(1, q^3)$. Let $G$ be the subgroup of projectivities of $\PG(U_1)$ isomorphic to $\PGL(2, q^3)$ induced by the matrices
\begin{align*}
& \begin{pmatrix}
a & b \\
c & d 
\end{pmatrix}
\otimes 
\begin{pmatrix}
a^q & b^q \\
c^q & d^q 
\end{pmatrix}
\otimes
\begin{pmatrix}
a^{q^2} & b^{q^2} \\
c^{q^2} & d^{q^2} 
\end{pmatrix}, & a,b,c,d \in \F_{q^3}, ad- bc \ne 0.
\end{align*}
Then $G$ fixes $\cO_1$ and acts 3-transitively on its points. Since 
\begin{align*}
\begin{pmatrix}
a & b \\
c & d 
\end{pmatrix}^t 
\begin{pmatrix}
0 & 1 \\
-1 & 0 
\end{pmatrix}
\begin{pmatrix}
a & b \\
c & d 
\end{pmatrix} = 
(ad-bc) 
\begin{pmatrix}
0 & 1 \\
-1 & 0 
\end{pmatrix},
\end{align*}
the symplectic polar space $\cW(7, q)$ of $\PG(U_1)$ with alternating bilinear form given by
\begin{align}
\begin{pmatrix}
0 & 1 \\
-1 & 0 
\end{pmatrix}
\otimes 
\begin{pmatrix}
0 & 1 \\
-1 & 0 
\end{pmatrix}
\otimes 
\begin{pmatrix}
0 & 1 \\
-1 & 0 
\end{pmatrix} \label{bil}
\end{align}
is stabilized by $G$. If $q$ is even, $G$ also preserves a hyperbolic quadric $\cQ^+(7, q)$ of $\PG(U_1)$. In this case $\cO_1$ is an ovoid of $\cQ^+(7, q)$, whereas if $q$ is odd, $\cO_1$ is a maximal partial ovoid of $\cW(7, q)$ \cite{K, C}; see also \cite{L, FKLS}. By Lemma~\ref{lemma1}, no four points of $\cO_1$ are contained in a plane, i.e., $\cO_1$ is a $4$-general set of $\PG(U_1)$. Furthermore, it is complete \cite{P}. 

The group $\PGL(2, q^3)$ acts transitively on the $q^2(q^4+q^2+1)$ $q$-order sublines of $\PG(1, q^3)$ and the canonical $q$-order subline of $\PG(1,q^3)$ corresponds to 
\begin{align}
\left\{P(1,t,t^2,t^3) = \left(1,t,t,t^2,t,t^2,t^2,t^3\right) \mid t \in \F_q\right\} \cup \left\{P(0,0,0,1) = (0,0,0,0,0,0,1)\right\}, \label{can}
\end{align}
that forms a twisted cubic. Such a twisted cubic is the complete intersection of $\cO_1$ with the solid of $\PG(U_1)$ induced by
\begin{align*}
\left\{ \left(a, b, b, c, b, c, c, d \right) : a, b, c, d \in \F_{q} \right\}.
\end{align*}
It follows that $\cO_1$ contains a distinguished set of $q^2(q^4+q^2+1)$ twisted cubics. Next we see that these are all the twisted cubics contained in $\cO_1$. For more information on twisted cubics the reader is referred to \cite[Chapter 21]{H}.

\begin{lemma}\label{lemma1}
For $q \ge 4$, any five distinct points of $\cO_1$ that are contained in a solid $\Sigma$, lie in a twisted cubic of $\cO_1$. Such a twisted cubic is the complete intersection of $\cO_1$ with $\Sigma$ and corresponds to a $q$-order subline of $\PG(1, q^3)$. 
\end{lemma}
\begin{proof}
We may assume without loss of generality that the five points are given by
\begin{align*}
& P(1,0,0,0), \; P(0,0,0,1), \; P(1,1,1,1), \; P\left(1,t_1,t_1^{q^2+q},t_1^{q^2+q+1}\right), \; P\left(1,t_2,t_2^{q^2+q},t_2^{q^2+q+1}\right),
\end{align*}
where $t_1, t_2 \in \F_{q^3} \setminus \{0, 1\}$, with $t_1 \ne t_2$. The fact that the five points are contained in a solid implies that there exist $\alpha, \beta, \gamma \in \F_q$, with $(\alpha, \beta, \gamma) \ne (0,0,0)$, such that 
\begin{align*}
& \alpha + \beta t_1 + \gamma t_2 = 0, \\
& \alpha + \beta t_1^{q^2+q} + \gamma t_2^{q^2+q} = 0.
\end{align*}
We want to show that $t_1, t_2 \in \F_q \setminus \{0, 1\}$. Assume by contradiction that $t_1 \notin \F_q$. Then $\gamma \ne 0$, otherwise either $(\alpha, \beta, \gamma) = (0,0,0)$ or $t_1 \in \F_q$. Hence 
\begin{align*}
& t_2 = - \frac{\alpha}{\gamma} - \frac{\beta}{\gamma} t_1, \\
& t_2^{q^2+q} = - \frac{\alpha}{\gamma} - \frac{\beta}{\gamma} t_1^{q^2+q}.
\end{align*}
It follows that
\begin{align}
\frac{\alpha}{\gamma} \left( \frac{\alpha}{\gamma} + 1 \right) + \frac{\beta}{\gamma} \left( \frac{\beta}{\gamma} + 1 \right) t_1^{q^2+q} + \frac{\alpha \beta}{\gamma} \left( t_1^{q^2} + t_1^{q} \right) = 0, \label{eq1}
\end{align}
and therefore
\begin{align*}
\frac{\alpha \beta}{\gamma} t_1^2 - \left( \frac{\alpha \beta}{\gamma} \left( t_1^{q^2} + t_1^{q} + t_1 \right) + \frac{\alpha}{\gamma} \left( \frac{\alpha}{\gamma}  + 1 \right) \right) t_1 - \frac{\beta}{\gamma} \left( \frac{\beta}{\gamma} + 1 \right) t_1^{q^2+q+1} = 0. 
\end{align*}
Set 
\begin{align*}
& A = \frac{\alpha \beta}{\gamma}, \;\; 
B = \frac{\alpha \beta}{\gamma} \left( t_1^{q^2} + t_1^{q} + t_1 \right) + \frac{\alpha}{\gamma} \left( \frac{\alpha}{\gamma}  + 1 \right), \;\;
C = \frac{\beta}{\gamma} \left( \frac{\beta}{\gamma} + 1 \right) t_1^{q^2+q+1}. & 
\end{align*}
Observe that $C \ne 0$. Indeed, if $\frac{\beta}{\gamma} = 0$, then by \eqref{eq1} either $\frac{\alpha}{\gamma} = 0$ or $\frac{\alpha}{\gamma} = -1$, that is $t_2 \in \{0, 1\}$, a contradiction. If $\frac{\beta}{\gamma} = -1$, then $\frac{\alpha}{\gamma} \ne 0$, otherwise $t_1 = t_2$, and hence \eqref{eq1} implies that $t_1^q + t_1^{q^2} \in \F_q$ and therefore $t_1 \in \F_q$, a contradiction. We infer that there exists a non-zero polynomial, namely $F(x) = A x^2 - B x - C$ with coefficients in $\F_q$ of degree at most two such that $F(t_1) = 0$. Hence $t_1 \in \F_{q^2} \cap \F_{q^3} = \F_q$, a contradiction.
\end{proof}

\begin{prop}\label{prop}
For $q \ge 4$, any twisted cubic contained in $\cO_1$ corresponds to a $q$-order subline of $\PG(1, q^3)$. 
\end{prop}
\begin{proof}
For any twisted cubic contained in $\cO_1$, take five of its points. By Lemma~\ref{lemma1}, such a twisted cubic corresponds to a $q$-order subline of $\PG(1, q^3)$.  
\end{proof}

\begin{lemma}\label{lem:seven-points}
If seven distinct points of $\cO_1$ are contained in a four-dimensional projective subspace of $\PG(U_1)$, then four of them are contained in a twisted cubic of $\cO_1$.
\end{lemma}
\begin{proof}
We may assume without loss of generality that the seven points of $\cO_1$ are given
\begin{align*}
& P(1,0,0,0), \; P(0,0,0,1), \; P(1,1,1,1), \; P\left(1,t_1,t_1^{q^2+q},t_1^{q^2+q+1}\right), \\
& P\left(1,t_2,t_2^{q^2+q},t_2^{q^2+q+1}\right), \; P\left(1,t_3,t_3^{q^2+q},t_3^{q^2+q+1}\right), \; P\left(1,t_4,t_4^{q^2+q},t_4^{q^2+q+1}\right),
\end{align*}
where $t_1, t_2, t_3, t_4$ are four pairwise distinct elements in $\F_{q^3} \setminus \{0, 1\}$. Assume that these seven points are in a $\PG(4, q)$ of $\PG(U_1)$, then 
\begin{align*}
\pi = \langle P(1,0,0,0), P(0,0,0,1), P(1,1,1,1) \rangle
\end{align*} 
is a plane of $\PG(4, q)$ and $P\left(1,t_i,t_i^{q^2+q},t_i^{q^2+q+1}\right) \notin \pi$, $i = 1,2,3,4$, since $\cO_1$ is a $4$-general set. Embed $\PG(4, q)$ in $\PG(4, q^3)$ and let $\bar{\pi}$ be the plane of $\PG(4, q^3)$ such that $\pi = \bar{\pi} \cap \PG(4, q)$.Then, by projecting the four points $P\left(1,t_i,t_i^{q^2+q},t_i^{q^2+q+1}\right)$, $i = 1,2,3,4$, from $\bar{\pi}$, we obtain four points $Q_i$, $i = 1,2,3,4$, on a line of $\PG(4, q^3)$ disjoint from $\bar{\pi}$. More explicitly, we can consider $Q_i$ be given by 
\begin{align*}
Q_i & = P\left(1,t_i,t_i^{q^2+q},t_i^{q^2+q+1}\right) - \left(1- t_i^{q^2+1}\right) P(1,0,0,0) - \left(t_i^{q^2+q+1} - t_i^{q^2+1}\right) P(0,0,0,1)  - t^{q^2+1} P(1,1,1,1) \\
 & = \left( 0, t_i^{q^2} - t_i^{q^2+1}, t_i^q - t_i^{q^2+1}, t_i^{q^2+q} - t_i^{q^2+1}, t_i - t_i^{q^2+1}, 0, t_i^{q+1} - t_i^{q^2+1}, 0 \right).
\end{align*}
Therefore 
\begin{align*}
\rk
\begin{pmatrix}
Q(t_1)\\
Q(t_2)\\
Q(t_3)\\
Q(t_4)
\end{pmatrix}
\le 2. 
\end{align*}
Set
\begin{align*}
& Q_1(t)=t^{q^2}-t^{q^2+1}, \; Q_2(t)=t^q-t^{q^2+1}, \; Q_3(t)=t^{q^2+q}-t^{q^2+1}, \; Q_4(t)=t-t^{q^2+1},\; Q_5(t)=t^{q+1}-t^{q^2+1}, & \end{align*}
and
\begin{align*}
& M_{\alpha\beta\gamma}(t_i,t_j,t_k) = \det
\begin{pmatrix}
Q_\alpha(t_i)&Q_\beta(t_i)&Q_\gamma(t_i)\\
Q_\alpha(t_j)&Q_\beta(t_j)&Q_\gamma(t_j)\\
Q_\alpha(t_k)&Q_\beta(t_k)&Q_\gamma(t_k)
\end{pmatrix}. & 
\end{align*}
Then 
\begin{align}
0 = & \; (t_1-t_4)(t_2-t_3)M_{245}(t_1,t_2,t_3)M_{123}(t_1,t_2,t_4) - (t_1-t_3)(t_2-t_4)M_{245}(t_1,t_2,t_4)M_{123}(t_1,t_2,t_3) & \nonumber \\
& -(t_1-t_2)(t_3-t_4)M_{234}(t_1,t_3,t_4)M_{125}(t_2,t_3,t_4) & \nonumber \\
= & \; \left( (t_1-t_3)^q(t_2-t_4)^q(t_1-t_4)(t_2-t_3)-(t_1-t_3)(t_2-t_4)(t_1-t_4)^q(t_2-t_3)^q \right) \left( \prod_{i=1}^4(t_i^q-t_i) \right) & \nonumber \\
= & \; (t_1-t_4)^q(t_2-t_3)^q(t_1-t_4)(t_2-t_3) \left( \frac{(t_1-t_3)^q(t_2-t_4)^q}{(t_1-t_4)^q(t_2-t_3)^q} - \frac{(t_1-t_3)(t_2-t_4)}{(t_1-t_4)(t_2-t_3)} \right) \left( \prod_{i=1}^4(t_i^q-t_i) \right). & \label{eq3}
\end{align}
Recall that four distinct points $(1, u), (1, v), (1, w), (1, z) \in \PG(1,q^3)$ lie in a subline $\PG(1,q)$ if and only if their cross-ratio 
\begin{align*}
\{u, v; w, z\} =\frac{(u-w)(v-z)}{(u-z)(v-w)} 
\end{align*}
belongs to $\F_q$. Hence from \eqref{eq3}, either $t_i \in \F_q$, for some $i \in \{1,2,3,4\}$, and the four points 
\begin{align*}
& P(1,0,0,0), \; P(0,0,0,1), \; P(1,1,1,1), \; P\left(1,t_i,t_i^{q^2+q},t_i^{q^2+q+1}\right)
\end{align*}
belong to the twisted cubic \eqref{can}, or 
\begin{align*}
\{t_1, t_2; t_3, t_4\}^q = \{t_1, t_2; t_3, t_4\}
\end{align*}
and the four points 
\begin{align*}
& P\left(1,t_1,t_1^{q^2+q},t_1^{q^2+q+1}\right), \; P\left(1,t_2,t_2^{q^2+q},t_2^{q^2+q+1}\right), \; P\left(1,t_3,t_3^{q^2+q},t_3^{q^2+q+1}\right), \; P\left(1,t_4,t_4^{q^2+q},t_4^{q^2+q+1}\right),
\end{align*}
lie on a twisted cubic contained in $\cO_1$.
\end{proof}

\section{$5$-general sets and $(6, 4)$-sets in $\PG(6, q)$}\label{pg6}

Here we adopt the same notation used in Section~\ref{pre}. The group $G$ has four orbits on both points and hyperplanes of $\PG(U_1)$ and the number of points that a hyperplane of $\PG(U_1)$ has in common with $\cO_1$ takes one of the following values: 
\begin{align*}
& 1, \; q^2-q+1, \; q^2+1, \; q^2+q+1.
\end{align*}
In particular, if $\perp$ denotes the symplectic polarity of $\PG(U_1)$ given by \eqref{bil}, then $|P^\perp \cap \cO_1| = 1$ if and only if $P \in \cO_1$, see \cite{FKLS}. 

\begin{prop}\label{prop2}
For $q \ge 4$, if a hyperplane $H$ of $\PG(U_1)$ contains a twisted cubic of $\cO_1$, then $|H \cap \cO_1| \in \left\{q^2+1, q^2+q+1\right\}$.
\end{prop}
\begin{proof}
By Proposition~\ref{prop}, we may assume that the hyperplane $H$ contains 
\begin{align*}
\left\{P(1,t,t^2,t^3) = (1,t,t,t^2,t,t^2,t^2,t^3) \mid t \in \F_q\right\} \cup \left\{P(0,0,0,1) = (0,0,0,0,0,0,1)\right\}.
\end{align*}
Let $H$ be given by 
\begin{align*}
& \lambda a + \mu b + \mu^q b^q + \mu^{q^2} b^{q^2} + \nu c + \nu^q c^q + \nu^{q^2} c^{q^2} + \delta d = 0, & \mbox{ for some } \lambda, \delta \in \F_q, \mu, \nu \in \F_{q^3}. 
\end{align*}
Then $\lambda = \delta = \mu + \mu^q + \mu^{q^2} = \nu + \nu^q + \nu^{q^2} = 0$. We claim that the number of elements $t \in \F_{q^3}$ satisfying  
\begin{align}
& \mu t + \mu^q t^q + \mu^{q^2} t^{q^2} + \nu t^{q^2+q} + \nu^q t^{q^2+1} + \nu^{q^2} t^{q+1} = 0, \label{eq2}
\end{align}
equals $q^2$ or $q^2+q$. Let $W$ be $4$-dimensional $\F_q$-vector space given by 
\begin{align*}
\left\{(t,t^q,t^{q^2}, z) \mid z \in \F_q, t \in \F_{q^3}\right\}.
\end{align*}
Then $\PG(W) \simeq \PG(3, q)$. Let
\begin{align*}
\Phi(t, z) = \nu t^q \, t^{q^2} + \nu^q t \, t^{q^2} + \nu^{q^2} t \, t^q + z \left( \mu t + \mu^q t^q + \mu^{q^2} t^{q^2} \right).
\end{align*}
Since $\Phi(t, z)^q = \Phi(t, z)$, the points of $\PG(W)$ satisfying $\Phi(t, z) = 0$ form a quadric $\cQ$ of $\PG(W)$. Then the number of elements $t \in \F_{q^3}$ satisfying \eqref{eq2} equals 
\begin{align*}
|\cQ \setminus \pi |, 
\end{align*}
where $\pi$ is the plane $z = 0$. Since $\mu + \mu^q + \mu^{q^2} = \nu + \nu^q + \nu^{q^2} = 0$, the quadric $\cQ$ contains the line $\left\{\left(t,t^q,t^{q^2}, z\right) \mid z, t \in \F_q \right\}$ and by \cite[Theorem 7.16]{H1}, $\cQ \cap \pi$ is a non-degenerate conic. Hence $\cQ$ is either a quadratic cone or a hyperbolic quadric of $\PG(W)$. The claim follows.
\end{proof}

Let $H_1$ be a hyperplane of $\PG(U_1)$ such that $|H_1 \cap \cO_1| = q^2-q+1$. In what follows we show that $H_1 \cap \cO_1$ is a $5$-general set of $H_1 \simeq \PG(6, q)$. 

\begin{cor}\label{cor2}
For $q \ge 4$, a twisted cubic of $\cO_1$ meets $H_1$ in at most three points.
\end{cor}
\begin{proof}
Assume by contradiction that a twisted cubic $\cC$ of $\cO_1$ has at least four points in common with $H_1$, then the solid spanned by $\cC$ is contained in $H_1$ and $\cC \subset H_1 \cap \cO_1$, contradicting Proposition~\ref{prop2}.
\end{proof}

\begin{theorem}
For $q \ge 4$, a solid of $H_1$ has at most four points in common with $\cO_1$.
\end{theorem}
\begin{proof}
Assume by contradiction that a solid of $H_1$ contains at least five distinct points of $\cO_1$. By Lemma~\ref{lemma1}, these five points are on a twisted cubic of $\cO_1$, which contradicts Corollary~\ref{cor2}.
\end{proof}

\begin{remark}\label{oss}
For $q \ge 4$, there exists a $4$-dimensional projective space of $H_1$ containing six points of $H_1 \cap \cO_1$, otherwise by projecting $H_1 \cap \cO_1$ we would obtain a $5$-general set of size $q^2-q$ in the quotient geometry, which is isomorphic to a $\PG(5, q)$. If $q > 5$, this contradicts Proposition~\ref{upper2}; if $q \in \{4, 5\}$, a direct check confirms the existence of a $4$-dimensional projective space of $H_1$ containing six points of $H_1 \cap \cO_1$. 
\end{remark}
Therefore the following holds.

\begin{theorem}\label{th1}
Let $H_1$ be a hyperplane of $\PG(U_1)$ such that $|H_1 \cap \cO_1| = q^2-q+1$, then $H_1 \cap \cO_1$ is a $5$-general set of $H_1 \simeq \PG(6, q)$, whenever $q \ge 4$. 
\end{theorem}

Next we show that a $4$-dimensional projective space of $H_1$ contains at most six points of $H_1 \cap \cO_1$, i.e., $H_1 \cap \cO_1$ is a $(6,4)$-set of $H_1 \simeq \PG(6, q)$. 

\begin{theorem}
For $q \ge 4$, a four-dimensional projective subspace of $H_1$ has at most six points in common with $\cO_1$.
\end{theorem}
\begin{proof}
Suppose, by contradiction, that there are seven distinct points of $H_1 \cap \cO_1$ contained in a four-dimensional projective subspace of $H_1$. By Lemma~\ref{lem:seven-points}, four of them lie on a twisted cubic contained in $\cO_1$, contradicting Corollary~\ref{cor2}. 
\end{proof}

By Remark~\ref{oss}, it is easy to see that, if $q \ge 4$, there exists a $5$-dimensional projective subspace of $H_1$ containing eight points of $H_1 \cap \cO_1$. Hence the following holds.

\begin{theorem}\label{th2}
Let $H_1$ be a hyperplane of $\PG(U_1)$ such that $|H_1 \cap \cO_1| = q^2-q+1$, then $H_1 \cap \cO_1$ is a $(6, 4)$-set of $H_1 \simeq \PG(6, q)$, whenever $q \ge 4$. 
\end{theorem}

\section{$4$-general sets and $(5, 3)$-sets in $\PG(5, q)$}\label{pg5}

Let $\cX$ be an $(r, s)$-set in $\PG(n, q)$ and let $P$ be a point of $\cX$. By projecting $\cX \setminus \{P\}$ from $P$ an $(r-1, s-1)$-set of $\PG(n-1, q)$ of size $|\cX| - 1$ arises. Hence, if $\cX = H_1 \cap \cO_1$ and $P \in \cX$, by Theorem~\ref{th1} and Theorem~\ref{th2}, the set obtained by projecting $\cX \setminus \{P\}$ from $P$, say $\cY$, is a $4$-general set and a $(5, 3)$-set in $\PG(5, q)$ of size $q^2-q$. More precisely, since $P^\perp \cap \cX = \{P\}$, then $|P^\perp \cap \cY| = 0$, where $P^\perp \cap \PG(5, q) \simeq \PG(4, q)$. Therefore the following holds true.

\begin{theorem}\label{th3}
In $\AG(5, q)$, $q \ge 4$, there exists a set of size $q^2-q$ that is a $4$-general set and a $(5, 3)$-set.
\end{theorem}

We remark that, to the best of our knowledge, the only $4$-general set in $\PG(5, q)$ of order $q^2$ known in the literature has been described by Cooperstein in \cite[Theorem~7.7]{Co}. In this section we provide a more explicit description of $\cX$ and $\cY$. This will allow us to obtain a $4$-general set of size $q^2-q+2$ in $\PG(5, q)$, $q \ge 4$.

In $\F_{q^{3}}^{6}$ consider the $6$-dimensional $\F_q$-subspace $V_0$ given by the set of vectors
\begin{align*}
\left\{\left(a, a^q, a^{q^2}, b, b^{q}, b^{q^2} \right) : a, b \in \F_{q^3} \right\}.
\end{align*}
Consider the Grassmannian of the planes of $\PG(V_0) \simeq \PG(5, q)$ obtained by selecting the following as Grassmann coordinates of a plane of $\PG(V_0) \simeq \PG(5, q)$:
\begin{align}
\begin{split}
& (123),(124),(125),(126),(134),(315),(136),(145),(146),(156),\\
& (234),(235),(236),(245),(426),(256),(345),(346),(356),(456). \label{coord}
\end{split}
\end{align} 
For more results on Grassmannians the reader is referred to \cite[Chapter 3]{HT}. The image under the Grassmann embedding of each of the $q^3+1$ members of the Desarguesian plane spread of $\PG(V_0)$ having as director spaces the lines $\U_1 \U_4$, $\U_2 \U_5$, $\U_3 \U_6$ has precisely eight non-zero entries. It is easily seen that the set so obtained spans the $7$-dimensional projective space $\PG(U_1)$ and it coincides with $\cO_1$. Based on this, now we provide an equivalent description of $\cO_1$. Set 
\begin{align*}
& V_1 = \left\{\left(a, a^q, a^{q^2}, a^{q^3}, a^{q^4}, a^{q^5} \right) : a \in \F_{q^6} \right\}, \;\; V_2 = \left\{\left(b, c, d, b^{q^3}, c^{q^3}, d^{q^3} \right) : b, c, d \in \F_{q^6} \right\}. 
\end{align*}
Then $\PG(5, q) \simeq \PG(V_1) \subset \PG(V_2) \simeq \PG(5, q^3)$. Let $\cG_{2, 5}$ be the Grassmannian of the planes of $\PG(V_1)$ with Grassmann coordinates as in \eqref{coord}. Then $\cG_{2, 5}$ lies in a (non-canonical) $19$-dimensional projective space, which (with a slight abuse of notation) will be denoted by $\PG(19, q)$. The image of the $q^3+1$ planes of the Desarguesian spread of $\PG(V_1)$ having as director spaces the line ${\U}_1 \U_4 \cap \PG(V_2)$, $\U_2 \U_5 \cap \PG(V_2)$, $\U_3 \U_6 \cap \PG(V_2)$ gives 
\begin{align*}
\cO_2 = \left\{ \left(1, t^{q^2}, t^q, t^{q^2+q}, t, t^{q^2+1}, t^{q+1}, t^{q^2+q+1} \right) : t \in \F_{q^6}, t^{q^3+1} = 1 \right\}.
\end{align*}
Hence $\cO_1$ and $\cO_2$ are projectively equivalent. Consider the Singer cyclic group of $\PG(V_1)$ whose Singer cycle is induced by
\begin{align*}
& D = \diag\left(\w, \w^q, \w^{q^2}, \w^{q^3}, \w^{q^4}, \w^{q^5}\right), 
\end{align*}   
where $\w$ is a primitive element of $\F_{q^6}$. Denote by $S$ the cyclic group of projectivities of $\PG(19, q)$ whose generator $\phi$ is induced by $\bigwedge^3(D)$, the third exterior power of $D$. The group $S$ fixes $\cG_{2, 5}$. In particular, it fixes three $5$-dimensional projective spaces $\Pi$, $\Pi_1$, $\Pi_2$, which span a $\PG(17, q)$, and a line $\ell$ of $\PG(19, q)$, where $\ell \cap \langle \Pi, \Pi_1, \Pi_2 \rangle_q = \emptyset$. The projectivity $\phi$ acts on $\ell$, $\Pi$, $\Pi_1$ and $\Pi_2$ as the map induced by
\begin{align*}
& D^{q^4+q^2+1}, \;\; D^{q^2+q+1}, \;\; D^{q^3+q+1}, \;\; D^{q^4+q+1}, 
\end{align*}
respectively. Set $\langle \ell, \Pi \rangle_q = \PG(U_2) \simeq \PG(7, q)$ and $\langle \ell, \Pi_1, \Pi_2 \rangle_q = \PG(U_3) \simeq \PG(13, q)$ where
\begin{align*}
& U_2 = \left\{u(a,b) = \left(b, b^{q^5}, a, b^{q^4}, b^{q},  a^q, b^{q^2}, b^{q^3} \right) : a \in \F_{q^2}, b \in \F_{q^6} \right\}, \\
& U_3 = \left\{ v(a,b,c)=\left(a, a^q, b, b^q, b^{q^2}, b^{q^3}, b^{q^4}, b^{q^5}, c, c^{q}, c^{q^2}, c^{q^3}, c^{q^4}, c^{q^5} \right) : a \in \F_{q^2}, b, c \in \F_{q^6} \right\}.
\end{align*}
Here $\ell$ and $\Pi$ have underlying vector space given by 
\begin{align*}
& \left\{u(a,0) : a \in \F_{q^2} \right\}, \quad  U = \left\{u(0,b) : b \in \F_{q^6} \right\}, 
\end{align*}
respectively. Let $\phi_i = \phi|_{\PG(U_i)}$, $i = 2, 3$. Then $\phi_2$, $\phi_3$ are induced by
\begin{align}
& u(a, b) \in U_2 \mapsto u\left(\w^{q^4+q^2+1} a, \w^{q^2+q+1} b \right) \in U_2, \nonumber \\
& v(a, b, c) \in U_3 \mapsto v\left(\w^{q^4+q^2+1} a, \w^{q^3+q+1} b, \w^{q^4+q+1} c\right) \in U_3, \label{phi3}
\end{align}
respectively. In this setting $\cO_2$ is contained in $\PG(U_2)$. Indeed, for $t \in \F_{q^6}$, with $t^{q^3+1} = 1$, let $b \in \F_{q^6}$ such that $b^q = b t$, and $a = b t^q$. Then 
\begin{align*}
& a^q = t^{q^2+1} b, \;\; a^{q^2} = a, & \\
& b^{q^2} = t^{q+1} b, \;\; b^{q^3} = t^{q^2+q+1} b, \;\; b^{q^4} = t^{q^2+q} b, \;\; b^{q^5} = t^{q^2} b, & 
\end{align*}
as required. Moreover, $\cO_2$ is left invariant by $\phi_2$. Define $\sigma$ as $\phi_2^{(q+1)(q^2+q+1)}$, that is induced by 
\begin{align*}
& u(a, b) \in U_2 \mapsto u\left( a, \w^{(q+1)(q^2+q+1)(q-q^4)} b \right) \in U_2.
\end{align*}
Note that $\gcd\left(q^2-q+1, q\right) = \gcd\left(q^2-q+1, q^2+q+1\right) = 1$ and hence 
\begin{align*}
\gcd\left(q^6-1, (q+1)(q^2+q+1)(q-q^4)\right) = (q-1)(q+1)(q^2+q+1). 
\end{align*}
Therefore $\w^{(q+1)(q^2+q+1)(q-q^4)}$ has order $\frac{q^6-1}{\gcd\left(q^6-1, (q+1)(q^2+q+1)(q-q^4)\right)} = q^2-q+1$. It follows that $\sigma$ acts as 
\begin{align*}
&\left(b, b^{q^5}, a, b^{q^4}, b^{q},  a^q, b^{q^2}, b^{q^3} \right) \in U_2 \mapsto \left(\eta b, \eta^{-q^2} b^{q^5}, a, \eta^{-q} b^{q^4}, \eta^q b^{q}, a^q, \eta^{q^2} b^{q^2}, \eta^{-1} b^{q^3} \right) \in U_2,
\end{align*}
where $\eta = \w^{(q+1)(q^2+q+1)}$ and $T = \langle \sigma \rangle$ is a cyclic group of order $q^2-q+1$.
The group $T$ fixes each of the hyperplanes of $\PG(U_2)$ through $\Pi$. Let $H_2$ be the hyperplane of $\PG(U_2)$ whose underlying vector space is given by 
\begin{align*}
\left\{u(a,b) : a \in \F_q, b \in \F_{q^6}\right\}. 
\end{align*}
Then 
\begin{align*}
H_2 \cap \cO_2 & =  \left\{ \left(1, t^{q^2}, t^q, t^{q^2+q}, t, t^{q^2+1}, t^{q+1}, t^{q^2+q+1} \right) : t \in \F_{q^6}, t^{q^2-q+1} = 1 \right\} \\
& =  \left\{ u(1, t^{-q}) = \left(t^{-q}, t^{-1}, 1, t^{q^2}, t^{-q^2}, 1, t, t^q \right) : t \in \F_{q^6}, t^{q^2-q+1} = 1 \right\} \\
& = \left\{ u(1, t) : t \in \F_{q^6}, t^{q^2-q+1} = 1\right\}, 
\end{align*}
and $T$ acts transitively on $H_2 \cap \cO_2$, where $|H_2 \cap \cO_2| = q^2-q+1$. Note that $\Pi \simeq \PG(5, q)$ is a hyperplane of $H_2 \simeq \PG(6, q)$ that is disjoint from $\cO_2$. Hence, by Theorem~\ref{th1} and Theorem~\ref{th2}, we have the following.
\begin{theorem}\label{th4}
In $\AG(6, q)$, $q \ge 4$, there exists a transitive set of size $q^2-q+1$ that is a $5$-general set and a $(6, 4)$-set.
\end{theorem}
By projecting from $u(1,1) \in H_2 \cap \cO_2$ the $q^2-q$ points of $(H_2 \cap \cO_2) \setminus \{u(1,1)\}$ to the hyperplane $\Pi \simeq \PG(5, q)$ we obtain the pointset $\cY$ of size $q^2-q$ given by
\begin{align*}
\left\{u(0, b-1) : b \in \F_{q^6}, b^{q^2-q+1} = 1, b \ne 1\right\}.
\end{align*}
Next we show that $\cY$ lies in the cone having as vertex a line and base the $q^2-q$ points of $\cQ^-(3, q) \setminus \cC$, where $\cQ^-(3, q)$ is an elliptic quadric and $\cC$ one of its conic sections.

Define 
\begin{align*}
& V = \left\{u(0, b) : b \in \F_{q^6} \setminus \{0\}, F(b) = 0\right\}, & \mbox{ where } F(X) = X^{q^2}-X^q+X.
\end{align*}
Since 
\begin{align*}
\begin{pmatrix}
0 & -1 \\
1 & 1
\end{pmatrix}
\begin{pmatrix}
0 & -1 \\
1 & 1
\end{pmatrix}^q
\ldots 
\begin{pmatrix}
0 & -1 \\
1 & 1
\end{pmatrix}^{q^5}
= \begin{pmatrix}
0 & -1 \\
1 & 1
\end{pmatrix}^6
= \begin{pmatrix}
1 & 0 \\
0 & 1
\end{pmatrix}, 
\end{align*}
by \cite[Theorem 7]{MS}, it follows that $\PG(V) \simeq \PG(1, q)$ is a line of $\Pi$. Projecting $u(0, x) \in \Pi \setminus \PG(V)$ from $\PG(V)$ we recover $u(0, x + V)$. By means of the isomorphism 
\begin{align*}
& u(0, x + V) \in U/V \mapsto u(0, F(x)) \in W, & \mbox{ where } W = \left\{u(0, F(b)) : b \in \F_{q^6} \right\},
\end{align*}
we have that $\PG(W) \simeq \PG(3, q)$ is the quotient geometry of $\Pi$ from $\PG(V)$.

\begin{lemma}\label{lemma2}
The following hold.
\begin{itemize}
\item[i)] $|\PG(V) \cap \cY| = 0$,
\item[ii)] $W = \left\{ u(0, z) : z \in \F_{q^6}, z+z^q = z^{q^3} + z^{q^4} \right\}$.
\end{itemize}
\end{lemma} 
\begin{proof}
Let $u(0, x) \in \cY$, then there exists $b \in \F_{q^6}$, with $b^{q^2-q+1} = 1$, $b \ne 1$, such that $x = b-1 \ne 0$. Hence $(x+1)^{q^2+1} = (x+1)^q$, that is 
\begin{align*}
& x^{q^2+1} + x^{q^2} - x^q + x = 0.
\end{align*}
Therefore 
\begin{align*}
& F(x) = - x^{q^2+1} \ne 0,
\end{align*}
and $|\PG(V) \cap \cY| = 0$.

Let $u(0, z) \in W$, then there exists $x \in \F_{q^6}$ such that $z = F(x) = x^{q^2}-x^q+x$. Hence 
\begin{align*}
& z^q = x^{q^3}-x^{q^2}+x^q, \quad z^{q^3} = x^{q^5}-x^{q^4}+x^{q^3}, \quad z^{q^4} = x-x^{q^5}+x^{q^4}.
\end{align*}
Therefore 
\begin{align*}
& z+z^q = z^{q^3} + z^{q^4}.
\end{align*}
and the result follows.
\end{proof}

Let
\begin{align*}
\Phi(z) = z^{q^2} \, z + z^{q^2} \, z^{q} + z^{q} \, z^{q^3}. 
\end{align*}
Assume that $z+z^q = z^{q^3} + z^{q^4}$. Then 
\begin{align*}
\Phi(z)^q & = z^{q^3} \, z^{q} + z^{q^3} \, z^{q^2} + z^{q^2} \, z^{q^4} & \\
& = z^{q^2} \, z + z^{q^2} \, z^{q} + z^{q} \, z^{q^3} & \\
& = \Phi(z). & 
\end{align*}
Hence the points of $\PG(W)$ satisfying $\Phi(z) = 0$ form a quadric $\cE$ of $\PG(W)$. Let $\pi$ be the plane of $\PG(W)$ whose underlying vector space is 
\begin{align*}
\left\{ u(0, z) : z \in \F_{q^3} \right\} \subset W.
\end{align*}
Then by \cite[Theorem 7.16]{H1}, $\cE \cap \pi$ is a non-degenerate conic, say $\cC$. 

\begin{theorem}\label{th5}
The following hold.
\begin{itemize}
\item[i)] The projection of a point of $\cY$ from $\PG(V)$ yields a point of $\cE \setminus \cC$.  
\item[ii)] The projection of $\cY$ from $\PG(V)$ to $\cE \setminus \cC$ is bijective.
\end{itemize}
\end{theorem}
\begin{proof}
Let $u(0, z) \in W$ be the projection of a point of $\cY$ from $\PG(V)$. Then 
\begin{align*}
& z = F(b-1) = F(b) -1 = b^{q^2}-b^q+b-1, 
\end{align*}
where $b \in \F_{q^6}$, with $b^{q^2-q+1} = 1$, $b \ne 1$. By using the fact that $b^{q^2} = b^{q-1}$, we get 
\begin{align*}
& z = F(b-1) = \frac{(b-b^q)(b-1)}{b}, & z^q = F(b-1)^q = \frac{(b-1)(b^q-1)}{b}, & \\
& z^{q^2} = F(b-1)^{q^2} = \frac{(b^q-1)(b^q-b)}{b^{q+1}}, & z^{q^3} = F(b-1)^{q^3} = \frac{(b-b^q)(b-1)}{b^{q+1}}. & 
\end{align*}
An easy calculation shows that $\Phi(z) = 0$ and hence the point $u(0, z)$ belongs to the quadric $\cE$. On the other hand the point does not lie on $\cC$, otherwise $z \in \F_{q^3}$, that is
\begin{align*}
0 & = z^{q^3} - z \\
& = F(b-1)^{q^3} - F(b-1) \\
& = \frac{(1-b^q)(b-b^q)(b-1)}{b^{q+1}},
\end{align*}
contradicting the fact that $b \ne 1$. Moreover, given $z = F(b-1)$, then $b$ is uniquely determined as 
\begin{align*}
\left( \frac{z}{z^{q^3}} \right)^{q^5} = \left( \frac{F(b-1)}{F(b-1)^{q^3}} \right)^{q^5} = \left(b^q\right)^{q^5} = b.
\end{align*}
Hence the projection of a point of $\cY$ from $\PG(V)$ is injective. Next we show it is also surjective. Let $u(0, z)$ be a point of $\cE \setminus \cC$. Set $\alpha = z - z^{q^3}$. Then 
\begin{align*}
\alpha^q = z^q - z^{q^4} = z^{q^3} - z = - \alpha,
\end{align*}
by Lemma~\ref{lemma2}. Furthermore $z^{q^4} = z^q + \alpha$, $z^{q^3} = z - \alpha$, $z^{q^5} = z^{q^2} - \alpha$. Therefore
\begin{align*}
z \, z^{q^2} \, z^{q^4} - z^q \, z^{q^3} \, z^{q^5} & = z \, z^{q^2} \left( z^q + \alpha \right) - z^q \left( z - \alpha \right) \left( z^{q^2} - \alpha \right) & \\
& = \alpha \left( z^{q^2} \, z + z^q \, z + z^{q^2} \, z^q - z^q \, \alpha \right) & \\
& = \alpha \left( z^{q^2} \, z + z^{q^2} \, z^{q} + z^{q} \, z^{q^3}  \right) & \\
& = 0. & 
\end{align*}
Define $b = \frac{z^{q^5}}{z^{q^2}}$. Then 
\begin{align*}
b^{q^2-q+1} = \frac{z^q \, z^{q^3} \, z^{q^5}}{z \, z^{q^2} \, z^{q^4}} = 1.
\end{align*} 
Since $z^{q} = z^{q^4} - \alpha$, $z = z^{q^3} + \alpha$, $z^{q^5} = z^{q^2} - \alpha$ and $\Phi(z)^q = 0$, we get
\begin{align*}
F(b-1) & = F(b) - 1 = \frac{z^q}{z^{q^4}} - \frac{z}{z^{q^3}} + \frac{z^{q^5}}{z^{q^2}} - 1 & \\
 & = -\alpha \left(\frac{1}{z^{q^2}} + \frac{1}{z^{q^3}} + \frac{1}{z^{q^4}} \right) & \\
 & = -\alpha \frac{z^{q^3} \, z^{q^4} + z^{q^3} \, z^{q^2} + z^{q^2} \, z^{q^4} }{z^{q^2} \, z^{q^3} \, z^{q^4}} & \\
 & = -\alpha \frac{z^{q^3} \left( z^{q^4} - z^{q} \right)}{z^{q^2} \, z^{q^3} \, z^{q^4}} & \\
 & = -\frac{\alpha^2}{z \, z^{q^2} \, z^{q^4}} z,
\end{align*}
where $\frac{\alpha^2}{z \, z^{q^2} \, z^{q^4}} \in \F_q$. Indeed, 
\begin{align*}
\left( \frac{\alpha^2}{z \, z^{q^2} \, z^{q^4}} \right)^q = \frac{(-\alpha)^2}{z^q \, z^{q^3} \, z^{q^5}} = \frac{\alpha^2}{z \, z^{q^2} \, z^{q^4}}.
\end{align*}
Hence $u(0, F(b-1))$ and $u(0, z)$ represent the same point of $\PG(W)$. The proof is now complete.
\end{proof} 
 
\begin{cor}
$\cE$ is an elliptic quadric of $\PG(W) \simeq \PG(3, q)$.
\end{cor} 

Recall that $\cY$ is a $(3, 2)$-set of $\Pi$. Hence a plane of $\Pi$ has at most three points in common with $\cY$. It follows from Theorem~\ref{th5} that no point of $\PG(V)$ lies on a line of $\Pi$ that is secant to $\cY$. On the other hand, no point of $\PG(V)$ lies on a plane that has three points in common with $\cY$, otherwise by projecting this plane, one gets a line of $\PG(W)$ with three points in common with $\cE \setminus \cC$, which is not the case. Therefore, by selecting two distinct points $P_1, P_2$ of $\PG(V)$ we obtain that $\cY \cup \{P_1, P_2\}$ is also a $(3, 2)$-set of $\Pi$. 

\begin{theorem}\label{th6}
$\cY \cup \{P_1, P_2\}$ is a $4$-general set of $\Pi \simeq \PG(5, q)$, $q \ge 4$, of size $q^2-q+2$.
\end{theorem}
 
\section{A construction of $4$-general sets in $\PG(13, q)$}\label{pg13}

With the same notation used in Section~\ref{pg5}, let $\langle \ell, \Pi_1, \Pi_2 \rangle_q = \PG(U_3) \simeq \PG(13, q)$ where
\begin{align*}
& U_3 = \left\{ v(a,b,c)=\left(a, a^q, b, b^q, b^{q^2}, b^{q^3}, b^{q^4}, b^{q^5}, c, c^{q}, c^{q^2}, c^{q^3}, c^{q^4}, c^{q^5} \right) : a \in \F_{q^2}, b, c \in \F_{q^6} \right\}.
\end{align*}
With a slight abuse of notation we denote by $v(a,b,c)$, for $(a,b,c)\neq (0,0,0)$, the point of $\PG(U_3)$ defined by the vector $v(a,b,c)$. The $5$-dimensional projective subspace consisting of the points $v(0,b,0)$, $b \in \F_{q^6} \setminus \{0\}$, (resp. $v(0,0,c)$, $c \in \F_{q^6} \setminus \{0\}$) is $\Pi_1$ (resp. $\Pi_2$), whereas the line formed by the points $v(a,0,0)$, $a \in \F_{q^2} \setminus \{0\}$ is $\ell$. 
Recall that $\phi_3$ is the projectivity of $\PG(U_3)$ induced by \eqref{phi3}. 

\begin{lemma} \label{lemma3}
\begin{enumerate}
    \item[(i)] $\langle \phi_3 \rangle$ is a group of order $\frac{q^6-1}{q-1}$.
    \item[(ii)] The group $\langle \phi_3 \rangle$ has a semiregular action on points of $\PG(U_3) \setminus \left( \Pi_1 \cup \Pi_2 \cup \ell \right)$. 
\end{enumerate}
\end{lemma}
\begin{proof}
$(i)$ Consider the projectivity $\phi_3^i$ associated with the linear map given by $v(a, b, c) \in U_3 \mapsto v\left(\w^{i(q^4+q^2+1)} a, \w^{i(q^3+q+1)} b, \w^{i(q^4+q+1)} c\right) \in U_3$. If $(a, b, c) \ne (0, 0, 0)$, then $\phi_3^i$ is the identity if and only if $\w^{(q-q^2)i} = 1$ and hence $\frac{q^6-1}{q-1} \mid i$, since $\gcd(q^6-1, q) = 1$. 

$(ii)$ Observe preliminarily that   
\begin{align*}
\gcd\left(q^3+q+1, \frac{q^6-1}{q-1}\right) = \gcd\left(q^4+q+1, \frac{q^6-1}{q-1}\right) = 
\begin{cases}
1 & \mbox{ if } q \not\equiv 1 \pmod{3}, \\
3 & \mbox{ if } q \equiv 1 \pmod{3}.
\end{cases}
\end{align*}
Assume first that $q \not\equiv 1 \pmod{3}$. It follows that the point $v(a',b',c')$ belongs to the orbit with representatives  
\begin{align*}
& v(a,1,c) & \mbox{ if } b' \ne 0, \\
& v(a,0,1) & \mbox{ if } b' = 0, c' \ne 0, \\    
& v(1,0,0) & \mbox{ if } b' = c' = 0.
\end{align*}
Moreover these representatives identify a unique point orbit of $\langle \phi_3 \rangle$ on $\PG(U_3)$. Indeed, assume that $v(a, 1, c)$ and $v(d,1,e)$ are in the same orbit. Then there exists $\lambda \in \F_q \setminus \{0\}$ such that
\begin{align*}
v(d,1,e) = \lambda v\left(\w^{i(q^4+q^2+1)} a, \w^{i(q^3+q+1)}, \w^{i(q^4+q+1)} c\right),
\end{align*} 
for some $i = 0, \dots, q^6-1$. In particular, $\w^{i(q^3+q+1)} \in \F_q \setminus \{0\}$, that is $\frac{q^6-1}{q-1} \mid i$ and $v(a, 1, c)$, $v(d,1,e)$ represent the same point. A similar argument holds for the remaining representatives. 

In the case when $q \equiv 1 \pmod{3}$, we have that the point $v(a',b',c')$ can be mapped to one of the following three points:
\begin{align*}
& v(a,1,c), v(a,\w,c),v(a,\w^2,c)  & \mbox{ if } b' \ne 0, \\
& v(a,0,1), v(a,0,\w), v(a,0,\w^2) & \mbox{ if } b' = 0, c' \ne 0,  
\end{align*}
whereas, as in the previous case, $v(a',0,0)$ and $v(1,0,0)$ belong to the same orbit. However these representatives identify a unique point orbit only if the point belongs to $\ell \cup \Pi_1 \cup \Pi_2$. Indeed, the projectivity induced by the map $\phi_3^{\frac{q^6-1}{3(q-1)}}$ sends the point $v(a, \w^r, c)$ to the point $v\left(a \w^{\frac{(q^6-1)(q^3+q)}{3}}, \w^r, c \w^{\frac{(q^6-1)q^3}{3}}\right)$ and the point $v(a,0,\w^r)$ to the point $v\left(a \w^{\frac{(q^6-1)(q^3+q)}{3}}, 0, \w^r\right)$, where $ \w^{\frac{q^6-1}{3}}$ is a root of $X^2+X+1$ and lies in $\F_q$. Therefore, if $\xi^2+\xi+1 = 0$, for a fixed $r \in \{0,1,2\}$, the points
\begin{align*}
& v(a, \w^r, c), v(\xi^2 a, \w^r, \xi c), v(\xi a, \w^r, \xi^2 c), \\
& v(a,0,\w^r), v(\xi^2 a,0,\w^r), v(\xi a,0,\w^r),
\end{align*}
are in the same orbits.
\end{proof}

Let us define the following set of points of $\PG(U_3)$:
\begin{align*}
\cV = \left\{ v\left(x^{q^4+q^2+1}, x^{q^3+q+1}, x^{q^4+q+1} \right) : x \in \F_{q^{6}} \setminus \{0\}\right\}, 
\end{align*}
i.e., $\cV$ is the $\langle \phi_3 \rangle$-orbit of $v(1,1,1)$. 

\begin{prop} 
$\cV$ is a $4$-general set of $\PG(U_3) \simeq \PG(13, q)$.
\end{prop}
\begin{proof}
After normalizing, a point of $\cV$ is represented by:
\begin{align*}
\left(1, y^{q^4+q^2+1}, y^{-q^3-q}, y, y^{-q^5-q^3}, y^{q^2}, y^{-q^5-q}, y^{q^4}, y^{-q}, y^{q^4+1}, y^{-q^3}, y^{q^2+1}, y^{-q^5}, y^{q^4+q^2} \right), 
\end{align*}
where $y \in \F_{q^6}$, with $y^{\frac{q^6-1}{q-1}} = 1$. Let $y_1, y_2, y_3$ be three distinct elements of $\F_{q^6} \setminus \{0, 1\}$, where $y^{\frac{q^6-1}{q-1}} = 1$. 
We claim that $\rk(M) = 4$, where
\begin{align*}
M = \begin{pmatrix}
1 & 1 & 1 & 1 & 1 & 1 & 1 & 1 \\
1 & y_1^{q^4+q^2+1} & y_1 & y_1^{q^2} & y_1^{q^4} & y_1^{q^4+1} & y_1^{q^2+1} & y_1^{q^4+q^2} \\
1 & y_2^{q^4+q^2+1} & y_2 & y_2^{q^2} & y_2^{q^4} & y_2^{q^4+1} & y_2^{q^2+1} & y_2^{q^4+q^2} \\
1 & y_3^{q^4+q^2+1} & y_3 & y_3^{q^2} & y_3^{q^4} & y_3^{q^4+1} & y_3^{q^2+1} & y_3^{q^4+q^2} 
\end{pmatrix}.
\end{align*}
Let $z_i = y_i -1$. Then 
\begin{align*}
& y_i^{q^2} -1= z_i^{q^2}, \quad y_i^{q^4}-1=z_i^{q^4}, \quad y_i^{q^2+1}-1=z_i+z_i^{q^2}+z_i^{q^2+1}, \quad y_i^{q^4+1}-1=z_i+z_i^{q^4}+z_i^{q^4+1}, \\
& y_i^{q^4+q^2}-1=z_i^{q^2}+z_i^{q^4}+z_i^{q^4+q^2}, \quad y_i^{q^4+q^2+1}-1=z_i+z_i^{q^2}+z_i^{q^4}+z_i^{q^2+1}+z_i^{q^4+1}+z_i^{q^4+q^2}+z_i^{q^4+q^2+1}.
\end{align*}
Therefore $\rk(M) = 4$ if and only if $\rk(N) = 3$, where
\begin{align*}
N = \begin{pmatrix}
z_1 & z_1^{q^2} & z_1^{q^4} & z_1^{q^2+1} & z_1^{q^4+1} & z_1^{q^4+q^2} & z_1^{q^4+q^2+1} \\
z_2 & z_2^{q^2} & z_2^{q^4} & z_2^{q^2+1} & z_2^{q^4+1} & z_2^{q^4+q^2} & z_2^{q^4+q^2+1} \\
z_3 & z_3^{q^2} & z_3^{q^4} & z_3^{q^2+1} & z_3^{q^4+1} & z_3^{q^4+q^2} & z_3^{q^4+q^2+1} 
\end{pmatrix}.
\end{align*}
If $\dim\left(\langle z_1, z_2, z_3 \rangle_{q^2}\right) = 3$, then $\rk(N) = 3$. In the case when $\dim\left(\langle z_1, z_2, z_3 \rangle_{q^2}\right) = 2$, we may assume $z_3 = \alpha z_1 + \beta z_2$, where $\alpha, \beta \in \F_{q^2}$, $\alpha \beta \ne 0$, then 
\begin{align*}
& z_3^{q^2+1} = \alpha^2 z_1^{q^2+1} + \beta^2 z_2^{q^2+1} + \alpha \beta \left(z_1 z_2^{q^2} + z_1^{q^2} z_2\right), \\
& z_3^{q^4+1} = \alpha^2 z_1^{q^4+1} + \beta^2 z_2^{q^4+1} + \alpha \beta \left(z_1 z_2^{q^4} + z_1^{q^4} z_2\right), \\
& z_3^{q^2+q^2} = \alpha^2 z_1^{q^4+q^2} + \beta^2 z_2^{q^4+q^2} + \alpha \beta \left(z_1^{q^2} z_2^{q^4} + z_1^{q^4} z_2^{q^2}\right). 
\end{align*}
Therefore $\rk(N) = 3$, since
\begin{align*}
\alpha \beta 
\det 
\begin{pmatrix}
z_1^{q^2+1} & z_1^{q^4+1} & z_1^{q^4+q^2} \\
z_2^{q^2+1} & z_2^{q^4+1} & z_2^{q^4+q^2} \\
z_1 z_2^{q^2} + z_1^{q^2} z_2 & z_1 z_2^{q^4} + z_1^{q^4} z_2 & z_1^{q^2} z_2^{q^4} + z_1^{q^4} z_2^{q^2} 
\end{pmatrix} = \alpha \beta \left(z_1 z_2^{q^2} - z_1^{q^2} z_2\right)^{q^4+q^2+1} \ne 0.
\end{align*}
If $\dim\left(\langle z_1, z_2, z_3 \rangle_{q^2}\right) = 1$, then let $\gamma_i \in \F_{q^2} \setminus \{0, 1\}$, $i=1, 2$, $\gamma_1 \ne \gamma_2$, such that $z_i = \gamma_i z_3$, $i=1, 2$. In this case
\begin{align*}
& z_i^{q^2+1} = \gamma_i^2 z_3^{q^2+1}, \quad z_i^{q^4+q^2+1} = \gamma_i^3 z_3^{q^4+q^2+1}, & i = 1,2 .
\end{align*}
Hence
\begin{align*}
\det 
\begin{pmatrix}
\gamma_1 z_3 & \gamma_1^2 z_3^{q^2+1} & \gamma_1^3 z_3^{q^4+q^2+1} \\
\gamma_2 z_3 & \gamma_2^2 z_3^{q^2+1} & \gamma_2^3 z_3^{q^4+q^2+1} \\
z_3 & z_3^{q^2+1} & z_3^{q^4+q^2+1} 
\end{pmatrix} = z_3^{q^4+2q^2+3} 
\det\begin{pmatrix}
\gamma_1 & \gamma_1^2 & \gamma_1^3 \\
\gamma_2 & \gamma_2^2 & \gamma_2^3 \\
1 & 1 & 1
\end{pmatrix} = \gamma_1 \gamma_2 (\gamma_2 - \gamma_1)(1 - \gamma_1)(1 - \gamma_2) \ne 0
\end{align*}
and $\rk(N) = 3$. Therefore no four points of $\cV$ can be contained in a plane.
\end{proof}

\begin{theorem}\label{th} 
In $\PG(13, q)$ there exists a transitive $(3, 2)$-set of size $\frac{q^6-1}{q-1}$.
\end{theorem}

\section{Conclusion}

Here we have presented constructions of the following:
\begin{enumerate}
\item[{\em i)}] a $(4, 3)$-set and a $(6, 4)$-set in $\AG(6, q)$, $q \ge 4$, of size $q^2-q+1$;
\item[{\em ii)}] a $(3, 2)$-set and a $(5, 3)$-set in $\AG(5, q)$, $q \ge 4$, of size $q^2-q$;
\item[{\em iii)}] a $(3, 2)$-set in $\PG(5, q)$, $q \ge 4$, of size $q^2-q+2$;
\item[{\em iv)}] a $(3, 2)$-set in $\PG(13, q)$ of size $\frac{q^6-1}{q-1}$.
\end{enumerate}
Except for {\em iv)}, all these examples have cardinalities matching the theoretical upper bound up to a constant factor. An $(r, s)$-set in $\PG(n, q)$ is said to be complete if it is not contained in a larger $(r, s)$-set. It would be interesting to determine whether the constructed sets are complete or not.   

\bigskip
{\footnotesize
\noindent\textit{Acknowledgments.}
The research was supported by the Italian National Group for Algebraic and Geometric Structures and their Applications (GNSAGA--INdAM).}

\end{document}